\documentclass{amsart}
\usepackage{amssymb, amsmath}
\usepackage{a4}

\def\id{\operatorname{Id}}

\newtheorem{theorem}{Theorem}[section]
\newtheorem{lemma}[theorem]{Lemma}

\begin{document}
\title[Higher-dimensional Osserman metrics]
{Higher-dimensional Osserman metrics with non-nilpotent Jacobi
operators}
\author{E. Calvi\~{n}o-Louzao, E. Garc\'{\i}a-R\'{\i}o, P. Gilkey, and  R. V\'{a}zquez-Lorenzo}
\address{C-L, G-R, V-L: Department of Geometry and Topology, Faculty of Mathematics,
University of Santiago de Compostela, 15782 Santiago de Compostela,
Spain}
\email{estebcl@edu.xunta.es, eduardo.garcia.rio@usc.es,\smallbreak\qquad\qquad\qquad\phantom{a.}
ravazlor@edu.xunta.es}
\address{G: Mathematics Department, University of Oregon, Eugene, Oregon 97403, USA }
\email{gilkey@uoregon.edu}
\thanks{Supported by projects MTM2009-07756 and INCITE09 207 151 PR  (Spain).}
\begin{abstract}
We exhibit Osserman metrics with non-nilpotent Jacobi
operators and with non-trivial Jordan normal form in neutral
signature $(n,n)$ for any $n\ge3$.
These examples admit a natural almost para-Hermitian structure and are semi  para-complex Osserman with non-trivial Jordan
normal form as well; they neither satisfy the third Gray identity nor are they integrable.
\end{abstract}
\maketitle

\section{Introduction}

A pseudo-Riemannian manifold $(M,g)$ is said to be \emph{Osserman}
if the eigenvalues of the Jacobi operators  ${}^g\mathcal{J}(X):Y\rightarrow {}^gR(Y,X)X$ are constant on the unit
pseudo-sphere bundles $S^\pm (TM,g)$. Any isotropic space is
Osserman and the converse is true in the Riemannian (dim $M\neq 16$)
\cite{Chi, Ni1,Ni2} and Lorentzian  \cite{BBG, GKVa} settings.
{However, there exist many  non-symmetric Osserman pseudo-Riemannian
metrics in other signatures (cf. \cite{GKV, Gilkey1} and the
references therein). Since the eigenvalue structure need not determine the conjugacy class of a
self-adjoint operator in the indefinite setting, a pseudo-Riemannian
manifold is called \emph{Jordan-Osserman} if the Jordan normal form
of the Jacobi operators is constant on $S^\pm (TM,g)$.
Osserman metrics are Einstein and thus of constant sectional curvature
in dimensions $2$ and $3$. The special significance of the four-dimensional case
relies on the fact that a four-dimensional algebraic curvature
tensor is Osserman if and only if it is Einstein and self-dual;
the classification of all
four-dimensional Osserman metrics of neutral signature $(2,2)$ is
almost complete \cite{BBR, RiEx, De,DGV1,
DR-GR-VL-06,GR-Gi-VA-VL,GVV,GV-GD}.

The situation is much more difficult in higher dimensions where only
some partial results are known \cite{GKV}.  The structure of a
Jordan-Osserman algebraic curvature tensor strongly depends on the
signature $(p,q)$ of the metric tensor. For example,
the Jacobi operators of a spacelike Jordan-Osserman algebraic
curvature tensor are necessarily diagonalizable whenever $p<q$ \cite{G-I}. In the neutral case ($p=q$), the Jordan
normal form can be arbitrarily complicated
\cite{G-I-2}}.  However in the geometric setting, less is known as with the exception of some six-dimensional
examples of Osserman metrics with non-nilpotent Jacobi operators \cite{RiEx}, all previously known examples of
Osserman metrics have either diagonalizable or nilpotent Jacobi
operators (see \cite{GKV, Gilkey1, G-I-S} and the references
therein). There are, of course, other natural operators beside the Jacobi operator that one could examine -- see, for example,
the discussion of the spectral geometry of the skew-symmetric curvature operator \cite{IP}.

The purpose of this paper is to investigate further the construction in
\cite{RiEx}, showing that  \emph{for any affine Osserman manifold $(M,D)$, the cotangent bundle $T^*M$
equipped with the modified Riemannian extension is an Osserman manifold
whose Jacobi operators are, in general, neither diagonalizable nor nilpotent}.

\subsection{Affine Osserman manifolds}
Let $(M,D)$ be an $n$-dimensional affine manifold, i.e., $D$ is a torsion-free connection on the tangent bundle of a smooth manifold $M$
of dimension $n$. Let ${}^D\!\!R(X,Y):=D_XD_Y-D_YD_X-D_{[X,Y]}$ be the associated curvature operator.
We say that $(M,D)$ is \emph{affine Osserman} if
the Jacobi operators are nilpotent \cite{GKVV}, i.e. $0$ is the only eigenvalue of ${}^D\!\!\mathcal{J}(\cdot)$ on $TM$.

There are corresponding local notions which are important; an affine manifold
$(M,D)$ is said to be {\it affine Osserman at $P\in M$} if
${}^D\mathcal{J}(\cdot)$ is nilpotent on $T_PM$. Similarly a
pseudo-Riemannian manifold $(M,g)$ is said to be {\it Osserman} at $P\in M$ if the
eigenvalues of ${}^g\mathcal{J}(\cdot)$ are constant on $S^\pm(T_PM,g)$. Clearly
$(M,D)$ is affine Osserman if and only if $(M,D)$ is affine Osserman at every
point $P\in M$. Similarly $(M,g)$ is Osserman if and only if $(M,g)$ is
Osserman at every point $P\in M$ and if the eigenvalue structure and eigenvalue multiplicities are independent of $P$.

\subsection{Riemannian extensions}

Let $N:=T^*M$ be the cotangent bundle of an  $n$-dimensional manifold
$M$, let $\sigma:N\rightarrow M$ be the natural projection, and let $Z(N)$ be the zero section. If $x=(x_1,\dots,x_n)$ are
local coordinates on $M$, let $x'=(x_{1'},\dots,x_{n'})$ be the associated dual coordinates on the fiber where we
expand a $1$-form $\omega$ as $\omega=x_{i'}dx^i$; we shall adopt the {\it Einstein} convention and sum over repeated indices
henceforth. The following natural distribution will play a crucial role in our analysis:
$$\mathcal{Y}:=\operatorname{Span}\{\partial_{x_{1'}},\dots,\partial_{x_{n'}}\}=\ker(\sigma_*)\,.$$

For each vector field $X=X^i\partial_{x_i}$ on $M$, the evaluation map $\iota
X(P,\omega)=\omega(X_P)$ defines a function on  $N$ which, in local coordinates, is given by
$$\iota X(x_i,x_{i'})=x_{i'}X^i\,.$$
Vector fields on $N$ are characterized by their action on
functions $\iota X$; the complete lift $X^C$ of a vector field  $X$
on $M$ to $N$ is characterized by the identity
$$X^C(\iota Z)=\iota [X,Z],\qquad  \mbox{for all $Z\in C^\infty (TM)$.}$$
Moreover,
since a $(0,s)$-tensor field on  $N$ is characterized by its
evaluation on complete lifts of vector fields on $M$, for each
tensor field  $S$ of  type $(1,1)$ on $M$, we define a $1$-form
$\iota S$  on  $N$ which is characterized by the identity
$$(\iota S)(X^C)=\iota(SX)\,.$$

Let $(M,D)$ be an affine manifold. The
\emph{Riemannian extension} $g_D$ is the pseudo-Riemannian metric
on  $N$ of neutral signature $(n,n)$ characterized by the
identity:
$$
g_D(X^C,Y^C)=-\iota(D_XY+D_YX)\,.
$$
If $u$ and $v$ are cotangent vectors, let $u\circ v:=\frac12(u\otimes v+v\otimes u)$. Expand
$$D_{\partial_{x_i}}\partial_{x_j}={}^D\Gamma_{ij}{}^\ell\partial_{x_\ell}$$
to define the Christoffel symbols  ${}^D\Gamma$ of $D$. One then has:
$$
g_D= 2\, dx^i\circ dx^{i'}-2x_{k'}{}^D\Gamma_{ij}{}^kdx^i\circ dx^j\,.
$$

Riemannian extensions were originally defined by Patterson and
Walker \cite{PWalker} and further investigated in relating
pseudo-Riemannian properties of  $N$ with the affine structure of
the base manifold $(M,D)$. Moreover, Riemannian extensions were also
considered in \cite{GKVV} in relation to Osserman manifolds (see
also \cite{D}).

The \emph{modified Riemannian extension} is the neutral signature metric
on $N$ defined by (see \cite{RiEx} for a more general
construction)
$$
g_N:=\iota \id\circ\iota \id + g_D\,.
$$
In a system of local coordinates one has
\begin{equation}\label{eqn-1}
g_N= 2\, dx^i\circ dx^{i'}+\{x_{i'} x_{j'}
-2x_{k'}{}^D\Gamma_{ij}{}^k\} dx^i\circ dx^j\,.
\end{equation}

The manifold $(N,g_N)$ is a Walker manifold where the parallel
degenerate distribution is in this instance given by $\mathcal{Y}$
\cite{W-50}. There is a canonical  almost para-Hermitian structure
$\mathfrak{J}$, i.e. a linear map of $TN$ so that
$\mathfrak{J}^2=\id$ and $\mathfrak{J}^*g_N=-g_N$, which will play a
crucial role in our analysis. In local coordinates, it is given by
\begin{equation}\label{eqn-2}
\mathfrak{J}:\partial_{x_i}\rightarrow\partial_{x_i}-\{x_{i'}x_{j'}-2x_{k'}{}^D\Gamma_{ij}{}^k\}\partial_{x_{j'}}
\quad\text{and}\quad\mathfrak{J}:\partial_{x_{i'}}\rightarrow-\partial_{x_{i'}}\,.
\end{equation}

The case that $D$ is flat is of particular interest. Let $\tilde{\mathbb{C}}P$ be para-complex projective space of constant para-holomorphic
sectional curvature $+1$. Then
\cite{RiEx}:

\begin{theorem}\label{thm-1.1}
If $D$ is flat, then $(N,g_N)$ is isomorphic
to $\tilde{\mathbb{C}}P$.
\end{theorem}

Para-complex projective space $\tilde{\mathbb{C}}P$ is Jordan-Osserman with diagonalizable Jacobi operators. Let
$g_{\tilde{\mathbb{C}}P}(\xi,\xi)=\pm1$. Then the eigenvalues of
${}^{g_{\tilde{\mathbb{C}}P}}\!\!\mathcal{J}(\xi)$ are
$\pm(0,1,\frac14)$ with multiplicities
$(1,1,2n-2)$, respectively.
If
$(M,D)$ is affine Osserman, then $g_N$ can be viewed as a
deformation of $g_{\tilde{\mathbb{C}}P}$. This introduces Jordan normal form into the Jacobi operator,
but does not change the eigenvalue structure in the affine Osserman context:

\begin{theorem}\label{thm-1.2}
Let $(M,D)$ be an affine manifold.
\begin{enumerate}
\item If $(M,D)$ is affine Osserman at $P\in M$, then $(N,g_N)$ is Osserman at any $Q\in\sigma^{-1}(P)$. The
 eigenvalues of ${}^{g_N}\!\!\mathcal{J}(\cdot)$ on $S^\pm(T_QN,g_N)$ are $\pm(0,1,\frac14)$ with multiplicities
$(1,1,2n-2)$, respectively.
\item If $(M,D)$ is affine Osserman, then $(N,g_N)$ is Osserman.
\end{enumerate}
\end{theorem}

Let $\mathcal{P}(N)$ be the bundle over $N$ of non-degenerate  $\mathfrak{J}$-invariant tangent $2$-planes. If
$\pi\in\mathcal{P}(N)$, choose
$\xi\in S^+(\pi,g_N)$
 and, following \cite{SV92}, define the {\it para-complex Jacobi} operator to be:
$${}^{g_N}\!\!\mathcal{J}(\pi):={}^{g_N}\!\!\mathcal{J}(\xi)-{}^{g_N}\!\!\mathcal{J}(\mathfrak{J}\xi);$$
this operator is independent of the particular $\xi$ chosen. Higher order Jacobi operators of this nature were first considered by
Stanilov and Videv \cite{SV92} in the real setting.
 One says that $(N,g_N,\mathfrak{J})$ is {\it semi  para-complex Osserman} if ${}^{g_N}\!\!\mathcal{J}(\cdot)$ has
constant eigenvalues on $\mathcal{P}(N)$; if additionally
${}^{g_N}\!\!\mathcal{J}(\pi)$ commutes with $\mathfrak{J}$ for all $\pi\in\mathcal{P}(N)$, then
$(N,g_N,\mathfrak{J})$ is said to be para-complex Osserman -- this implies $D$ is flat by 
Theorem~\ref{thm-1.6} so this
condition is not particularly interesting in the setting we are considering.

\begin{theorem}\label{thm-1.3}
Let $(M,D)$ be an affine manifold. Let  $\pi\in\mathcal{P}(N)$. The
eigenvalues of ${}^{g_N}\!\!\mathcal{J}(\pi)$ are $(1,\frac12)$ with multiplicities $(2,2n-2)$, respectively, and any
Jordan block for
${}^{g_N}\!\!\mathcal{J}(\pi)$ has size at most $2\times 2$; $(N,g_N,\mathfrak{J})$ is semi  para-complex Osserman.
\end{theorem}

These examples provide genuinely new phenomena. The following result shows that Jordan normal form of ${}^{g_N}\!\!\mathcal{J}$ can be quite
complicated; it also shows that $(N,g_N)$ need not be Jordan-Osserman:
\begin{theorem}\label{thm-1.4}
Let $r\ge2$ and let $U$ be an $r\times r$ lower triangular matrix. There exists an affine Osserman manifold $(M,D)$ of
dimension $r+1$, there exists
$Q\in Z(N)$, and there exist
$\xi_i\in S^+(T_QN,g_N)$ for $i=1,2$ so that:
\begin{enumerate}
\item ${}^{g_N}\!\!\mathcal{J}(\xi_1)$ is diagonalizable.
\item Relative to a suitable basis for $T_QN$,
$$
\textstyle{}^{g_N}\!\!\mathcal{J}(\xi_2)=0\cdot\id_1\oplus
1\cdot\id_1\oplus(\frac14\cdot\id_r+U)\oplus(\frac14\cdot\id_r+U^t)\,.
$$
\end{enumerate}\end{theorem}

There also are non-trivial examples in the para-complex setting:
\begin{theorem}\label{thm-1.5}
Let $n\ge3$. There exists an affine Osserman manifold $(M,D)$ of dimension $n$ so that $(N,g_N,\mathfrak{J})$ is not Jordan
semi  para-complex Osserman, and so that the para-complex Jacobi operators are not always
diagonalizable.
\end{theorem}

One says that an almost para-Hermitian manifold $(A,g_A,\mathfrak{J})$ {\it satisfies the third Gray identity} if
\begin{equation}\label{eqn-3}
{}^{g_A}\!R(X,Y,Z,W)={}^{g_A}\!R(\mathfrak{J}X,\mathfrak{J}Y,\mathfrak{J}Z,\mathfrak{J}W)\quad\text{for
all}\quad X,Y,Z,W\,.
\end{equation}
An almost para-Hermitian manifold
$(A,g_A,\mathfrak{J})$ is {\it integrable} if there exist local
coordinates
$(u_1,\dots,u_n,v_1,\dots,v_n)$
centered at any given point of $A$ so  that
$$\mathfrak{J}\partial_{u_i}=\partial_{v_i}\quad\text{and}\quad\mathfrak{J}\partial_{v_i}=\partial_{u_i}$$
or, equivalently \cite{CMMS03}, if the Nijenhuis
tensor
$N_{{\mathfrak{J}}}$ vanishes where
\begin{equation}\label{eqn-4}
N_{{\mathfrak{J}}}(X,Y):=[X,Y]-{\mathfrak{J}}[{\mathfrak{J}}X,Y]
-{\mathfrak{J}}[X,{\mathfrak{J}}Y]+[{\mathfrak{J}}X,{\mathfrak{J}}Y]\,.
\end{equation}
\begin{theorem}\label{thm-1.6}
Let $(M,D)$ be an affine manifold. The following conditions are equivalent:
\begin{enumerate}
\item $(M,D)$ is flat.
\item $(N,g_N,\mathfrak{J})$ is integrable.
\item $(N,g_N,\mathfrak{J})$ satisfies the third Gray identity.
\item $\mathfrak{J}{}^{g_N}\!\!\mathcal{J}(\pi)={}^{g_N}\!\!\mathcal{J}(\pi)\mathfrak{J}$ for all $\pi\in\mathcal{P}(N)$.
\end{enumerate}
\end{theorem}

Theorem \ref{thm-1.2} was first
discovered in low dimensions using a computer assisted  calculation; subsequently the general case was
derived. Here is a brief outline to the paper. In Section
\ref{sect-2}, we prove Theorem \ref{thm-1.2} and Theorem \ref{thm-1.3}. In Section
\ref{sect-3}, we construct various examples to demonstrate Theorem \ref{thm-1.4} and Theorem
\ref{thm-1.5}. We conclude the paper in Section
\ref{sect-4} by establishing Theorem \ref{thm-1.6}.

\section{The eigenvalue structure}\label{sect-2}
We begin our discussion with a technical result. Although well known, we include the proof to keep our discussion as
self-contained as possible.  If
$D$ is an arbitrary connection on
$TM$, the {\it torsion tensor} $\mathcal{T}\in\Lambda^2(T^*M)$ is defined by:
$$\mathcal{T}(X,Y):=D_XY-D_YX-[X,Y]\,.$$

\begin{lemma}\label{lem-2.1}
Let $D$ be an arbitrary connection
on $TM$. Let $P\in M$. The following conditions are equivalent:
\begin{enumerate}
\item There exist local coordinates
$x=(x_1,\dots, x_n)$ centered at $P$ so ${}^D\Gamma(P)=0$.
\item  The torsion tensor $\mathcal{T}=0$ vanishes at $P$.
\end{enumerate}
\end{lemma}

\begin{proof} Let $x=(x_1,\dots, x_n)$ be a system of local coordinates on $M$. The torsion tensor $\mathcal{T}$ vanishes at $P$ if and
only if
${}^D\Gamma_{ij}{}^k(P)={}^D\Gamma_{ji}{}^k(P)$. In particular, if there exists a coordinate system where ${}^D\Gamma(P)=0$, then necessarily
$\mathcal{T}$ vanishes at $P$. Thus Assertion (1) implies Assertion (2). Conversely, assume that Assertion (2) holds.
Define a new system of coordinates by
setting:
$$z_i=x_i+{\textstyle\frac12}a_{ijk}x_jx_k$$
where $a_{ijk}=a_{ikj}$ remains to be chosen. As
$\partial_{x_j}=\partial_{z_j}+a_{l ji}x_i\partial_{z_l}$,
$$D_{\partial_{x_i}}\partial_{x_j}(0)=D_{\partial_{z_i}}\partial_{z_j}(0)+a_{l ji}\partial_{z_l}(0)\,.$$
Assertion (1) now follows by setting
$a_{l ij}:={}^D\Gamma_{ij}{}^l$; the fact that $a_{l ij}=a_{l ji}$
is exactly the assumption that $D$ is torsion-free at
$P$.
\end{proof}

 Let $(M,D)$ be an affine manifold, let $Q\in N=T^*M$, and let $P:=\sigma Q\in M$. As $D$ is torsion-free, we may apply Lemma
\ref{lem-2.1} to make a change of coordinates on
$M$ so that ${}^D\Gamma(P)=0$. Let
${}^{g_{\tilde{\mathbb{C}}P}}\!R$ be
the curvature tensor of the metric
$$g_{\tilde{\mathbb{C}}P}:=2dx^i\circ dx^{i'}+x_{i'} x_{j'}
dx^i\circ dx^j\,.
$$

This metric is {\bf not} invariantly defined but depends on the
coordinates chosen. We note $g_N(Q)=g_{\tilde{\mathbb{C}}P}(Q)$. We
set
${}^2R:={}^{g_N}\!R-{}^{g_{\tilde{\mathbb{C}}P}}\!R$.
Let
$\{{}^{g_N}\!\mathcal{J},{}^{g_{\tilde{\mathbb{C}}P}}\!\mathcal{J},
{}^{2}\!\mathcal{J},{}^{D}\!\mathcal{J},{}^{g_D}\!\mathcal{J}\}$
be the associated Jacobi operators.
\begin{lemma}\label{lem-2.2}
Let $(M,D)$ be an affine manifold, let $Q\in N$, and let $P=\sigma Q$. Let
$\xi\in T_QN$, and let $a=\sigma_*(\xi)\in T_PM$. Then we
have, relative to the natural coordinate frame
$\{\partial_{x_1},\dots,\partial_{x_n},\partial_{x_{1^\prime}},\dots,\partial_{x_{n^\prime}}\}$
for $TN$, that:
$$
{}^{2}\!\mathcal{J}(\xi)
=\left(\begin{array}{ll}
{}^{D}\!\!\mathcal{J}(a)&0\\\star&{}^{D}\!\!\mathcal{J}(a)^t
\end{array}\right)\,.
$$
\end{lemma}

\begin{proof} Decompose
${}^{g_N}\!R={}^{g_{\tilde{\mathbb{C}}P}}\!
R+{}^{g_D}\!R+{}^{\mathcal{E}}\!R$ where ${}^{\mathcal{E}}\!R$ is an
additional term measuring the interactions between the metrics
$g_{\tilde{\mathbb{C}}P}$ and $g_D$ in the combined metric $g_N$ of
Equation (\ref{eqn-1}). By \cite{BGGNV} (page 54 Equation (3.6)),
${}^{g_D}\!\!\mathcal{J}(\xi)$ has the form given in the
Lemma. We will complete the proof by showing that the additional
interaction terms define a Jacobi operator with
${}^{\mathcal{E}}\!\mathcal{J}(\xi):\operatorname{Span}\{\partial_{x_i}\}\rightarrow\mathcal{Y}$, i.e.
$$
{}^{\mathcal{E}}\!\mathcal{J}(\xi)=\left(\begin{array}{ll}0&0\\\star&0\end{array}\right)\,.
$$

We use dimensional analysis.  Define:\def\pdeg{\operatorname{Deg}}
$$\begin{array}{lll}
\pdeg(x_i)=-1,&\pdeg(dx^i)=-1,&\pdeg(\partial_{x_i})=+1,\\
\pdeg(x_{i^\prime})=+1,&\pdeg(dx^{i^\prime})=+1,&\pdeg(\partial_{x_{i^\prime}})=-1,\\
\pdeg({}^D\Gamma_{\star\star}{}^\star)=+1\,.
\end{array}$$
We consider the rescaling $x_i\rightarrow c^{-1}x_i$; this induces a dual rescaling
$x_{i^\prime}\rightarrow cx_{i^\prime}$. If $\Theta$ is a tensor of degree $k$, then
$\Theta\rightarrow c^k\Theta$ under this rescaling. For example, the metrics
$g_{\tilde{\mathbb{C}}P}$, $g_D$, and
$g_{N}$ are all homogeneous of degree 0; thus they are invariant under this rescaling.

It is clear that the Christoffel symbols of the first kind decouple:
$${}^{g_N}\!\Gamma_{\star\star\star}={}^{g_{\tilde{\mathbb{C}}P}}\!\Gamma_{\star\star\star}
+{}^{g_D}\!\Gamma_{\star\star\star}\,.$$
Since ${}^D\!\Gamma$ vanishes
at $P$, we must have at least one $\partial_{x_i}$ derivative of
${}^D\!\Gamma$ in computing ${}^{g_N}\!R$; thus any variable involving
${}^D\!\Gamma$ has degree at least $+2$. In raising indices in the
metric $g_N$ rather than in the metric $g_D$, we must take into
consideration the $x_{i^\prime}x_{j^\prime}dx^i\circ dx^j$ term;
thus interactions of this form involving ${}^D\Gamma$ contribute
terms of degree at least $2+2=4$ to ${}^{\mathcal{E}}R$. We also
have interaction terms which are bilinear in
${}^{g_{\tilde{\mathbb{C}}P}}\!\Gamma_{\star\star\star}$
and
${}^{g_D}\!\Gamma_{\star\star\star}$
after an index is raised in each factor. Such terms are at least
quadratic in $\{x_{i^\prime}\}$ and
linear in $\partial_{x_i}{}^D\Gamma$ and consequently have total
degree at least $+4$. We therefore conclude that any monomial of the
interaction tensor
${}^{\mathcal{E}}\!R$ has total
degree at least $+4$. The degree of
${}^{\mathcal{E}}\!R_{w_1w_2w_3}{}^{w_4}$ is $\pm1\pm1\pm1\pm1$; such
a term has degree at most $+4$ and the degree is exactly $+4$ if and
only if $w_1\in\{1,\dots,n\}$,
$w_2\in\{1,\dots,n\}$,
$w_3\in\{1,\dots,n\}$, and if
$w_4\in\{1^\prime,\dots,n^\prime\}$.
Consequently ${}^{\mathcal{E}}\!R$ defines a Jacobi operator mapping
$\operatorname{Span}\{\partial_{x_i}\}$ to $\mathcal{Y}$.
\end{proof}

Let $\xi\in S^+(T_QN,g_N)$ and let $\xi_1:=\mathfrak{J}\xi\in S^-(T_QN,g_N)$. Let $E_\lambda(\xi)$ (resp. $E_\lambda(\xi_1)$)
be the eigenspaces of ${}^{g_{\tilde{\mathbb{C}}P}}\!\!\mathcal{J}(\xi)$ (resp.
${}^{g_{\tilde{\mathbb{C}}P}}\!\!\mathcal{J}(\xi_1)$) for the eigenvalue $\lambda\in\{0,1,\frac14\}$ (resp. for
$\lambda\in\{0,-1,-\frac14\}$). Set
\begin{equation}\label{eqn-5}
\begin{array}{l}
E_0(\xi)=\xi\cdot\mathbb{R}=E_{-1}(\xi_1),\qquad\qquad
E_1(\xi)=\xi_1\cdot\mathbb{R}=E_0(\xi_1),\vphantom{\vrule height 11pt}\\
E_{\frac14}(\xi)=\{E_0(\xi)\oplus E_1(\xi)\}^\perp=\{E_{-1}(\xi_1)\oplus
E_0(\xi_1)\}^\perp=E_{-\frac14}(\xi_1),\vphantom{\vrule height 11pt}\\
\mathcal{S}(\xi):=\mathcal{Y}\cap E_{\frac14}(\xi),\qquad\qquad\qquad\phantom{a}
\mathcal{U}(\xi):=E_0(\xi)\oplus\mathcal{Y}\,.\vphantom{\vrule height 11pt}
\end{array}\end{equation}
We then have $T_QN=E_0(\xi)\oplus E_1(\xi)\oplus
E_{\frac14}(\xi)$.

\begin{lemma}\label{lem-2.3}
Let $(M,D)$ be an affine manifold. Let $Q\in N$. Let
$\xi\in S^+(T_QN,g_N)$.
\begin{enumerate}
\item $\mathcal{Y}=(\xi_1-\xi)\cdot\mathbb{R}+\mathcal{S}(\xi)$.
\item ${}^{2}\!\mathcal{J}(\xi)\mathcal{Y}\subset\mathcal{S}(\xi)$.
\item  ${}^{g_{\tilde{\mathbb{C}}P}}\!\!\mathcal{J}(\xi)\mathcal{U}(\xi)\subset \mathcal{U}(\xi)$ and
${}^{2}\!\mathcal{J}(\xi)\mathcal{U}(\xi)\subset \mathcal{U}(\xi)$.
\end{enumerate}\end{lemma}
\begin{proof}
Equation (\ref{eqn-2}) implies $\xi_1-\xi\in\mathcal{Y}$. We can choose an orthonormal basis
for
$E_{\frac14}(\xi)$ of the form
$$\{e_1^+,\dots,e_{n-1}^+,\mathfrak{J}e_1^+,\dots,\mathfrak{J}e_{n-1}^+\}$$
where the $e_i^+$ are spacelike and the  $\mathfrak{J}e_i^+$ are timelike. We prove Assertion (1) by noting
that we have the following basis for $\mathcal{Y}$:
$$\{\xi-\xi_1,e_1^+-\mathfrak{J}e_1^+,\dots,e_{n-1}^+-\mathfrak{J}e_{n-1}^+\}\,.$$

Suppose that Assertion (2) fails. We argue for a contradiction. Choose
$\eta\in\mathcal{Y}$ so that ${}^{2}\!\mathcal{J}(\xi)\eta\notin\mathcal{S}(\xi)$. By Lemma \ref{lem-2.2},
${}^{2}\!\mathcal{J}(\xi)\eta\in\mathcal{Y}$. By Assertion (1), there exists
$c\ne0$ so that
$${}^{2}\!\mathcal{J}(\xi)\eta=c(\xi-\xi_1)+\eta_1\quad\text{for}\quad\eta_1\in\mathcal{S}(\xi)\,.$$
Thus
$c\xi\in E_1(\xi)+\operatorname{Range}({}^{2}\!\mathcal{J}(\xi))+E_{\frac14}(\xi)\subset E_0(\xi)^\perp$
which is false; this contradiction establishes Assertion (2). To prove Assertion (3), express:
\begin{equation}\label{eqn-6}
\mathcal{U}(\xi)=E_0(\xi)\oplus\mathcal{Y}=\xi\cdot\mathbb{R}\oplus(\xi_1-\xi)\cdot\mathbb{R}\oplus\mathcal{S}(\xi)\\
=\xi\cdot\mathbb{R}\oplus\xi_1\cdot\mathbb{R}\oplus\mathcal{S}(\xi)\,.
\end{equation}
As ${}^{g_{\tilde{\mathbb{C}}P}}\!\!\mathcal{J}(\xi)\xi=0$, as ${}^{g_{\tilde{\mathbb{C}}P}}\!\!\mathcal{J}(\xi)\xi_1=\xi_1$, and
as $S(\xi)\subset E_{\frac14}(\xi)$, ${}^{g_{\tilde{\mathbb{C}}P}}\!\!\mathcal{J}(\xi)$ preserves $\mathcal{U}(\xi)$. As
${}^{2}\!\mathcal{J}(\xi)\xi=0$ and as ${}^{2}\!\mathcal{J}(\xi)\mathcal{Y}\subset\mathcal{Y}$, ${}^{2}\mathcal{J}(\xi)$
preserves
$\mathcal{U}(\xi)$ as well.
\end{proof}

We examine the eigenvalue structure:
\begin{lemma}\label{lem-2.4}
Let $(M,D)$ be an affine manifold. Let $Q\in N$. Let
$\xi\in S^+(T_QN,g_N)$. Assume $(M,D)$ is affine Osserman at $P=\sigma(Q)$. If
there is $0\ne\eta\in T_QN\otimes_{\mathbb{R}}\mathbb{C}$ with
${}^{g_N}\mathcal{J}(\xi)\eta=\mu\eta$, then:
\begin{enumerate}\item
If $\eta\not\in \mathcal{U}(\xi)\otimes_{\mathbb{R}}\mathbb{C}$, then $\mu=\frac14$.
\item If $\eta\in \mathcal{U}(\xi)\otimes_{\mathbb{R}}\mathbb{C}$ and if
$\eta\not\in\mathcal{S}(\xi)\otimes_{\mathbb{R}}\mathbb{C}$, then $\mu=0$ or $\mu=1$.
\item If $\eta\in \mathcal{S}(\xi)\otimes_{\mathbb{R}}\mathbb{C}$, then $\mu=\frac14$.
\item $\operatorname{Spec}\{{}^{g_N}\!\!\mathcal{J}(\xi)\}\subset\{0,1,\frac14\}$.
\end{enumerate}\end{lemma}

\begin{proof} By Lemma \ref{lem-2.2}, ${}^{2}\!\mathcal{J}(\xi)$ is
nilpotent since $(M,D)$ is affine Osserman at $P$.
By Lemma \ref{lem-2.3}, $\mathcal{U}(\xi)$ is preserved by
${}^{g_{\tilde{\mathbb{C}}P}}\!\!\mathcal{J}(\xi)$
and by ${}^{2}\!\mathcal{J}(\xi)$. Thus,
there are induced operators
${}^{g_{\tilde{\mathbb{C}}P}}\!\!\tilde{\mathcal{J}}(\xi)$, ${}^2\!\tilde{\mathcal{J}}(\xi)$, and
${}^{g_N}\!\!\tilde{\mathcal{J}}(\xi)={}^{g_{\tilde{\mathbb{C}}P}}\!\!\tilde{\mathcal{J}}(\xi)+{}^2\!\tilde{\mathcal{J}}(\xi)$
on the quotient space:
$$\mathcal{V}(\xi):=\{T_QN/\mathcal{U}(\xi)\}\otimes_{\mathbb{R}}\mathbb{C}\,.$$
If $\eta\notin \mathcal{U}(\xi)\otimes_{\mathbb{R}}\mathbb{C}$, then $\tilde\eta\in\mathcal{V}(\xi)$,
$\tilde\eta\ne0$ and ${}^{g_N}\!\tilde{\mathcal{J}}(\xi)\tilde\eta=\mu\tilde\eta$. By Equation (\ref{eqn-6}),
$$\mathcal{V}(\xi)=\{E_{\frac14}(\xi)/\mathcal{S}(\xi)\}\otimes_{\mathbb{R}}\mathbb{C}\,.$$
Consequently, ${}^{g_{\tilde{\mathbb C}P}}\!\!\tilde{\mathcal{J}}(\xi)=\frac14\id$. Since
${}^2\!\tilde{\mathcal{J}}(\xi)$ is nilpotent and
${}^{g_N}\!\!\tilde{\mathcal{J}}(\xi)=\frac14\id+{}^2\!\tilde{\mathcal{J}}(\xi)$,
${}^{g_N}\!\!\tilde{\mathcal{J}}(\xi)$ has only the eigenvalue $\frac14$. Thus $\mu=\frac14$. This
establishes Assertion (1).

To prove Assertion  (2),  suppose there exists $0\neq\eta\in\mathcal{U}(\xi)\otimes_{\mathbb{R}}\mathbb{C}$ such
that $\eta\notin\mathcal{S}(\xi)\otimes_{\mathbb{R}}\mathbb{C}$ and
${}^{g_N}\!\!\mathcal{J}(\xi)\eta=\mu\eta$. By Lemma \ref{lem-2.3}, $\mathcal{S}(\xi)$ is
preserved by ${}^{g_{\tilde{\mathbb{C}}P}}\!\!\mathcal{J}(\xi)$ and
${}^{2}\!\mathcal{J}(\xi)$. Thus there are
induced operators that  we again
denote by
${}^{g_{\tilde{\mathbb{C}}P}}\!\!\tilde{\mathcal{J}}(\xi)$, ${}^2\!\tilde{\mathcal{J}}(\xi)$, and
${}^{g_N}\!\!\tilde{\mathcal{J}}(\xi)={}^{g_{\tilde{\mathbb{C}}P}}\!\!\tilde{\mathcal{J}}(\xi)+{}^2\!\tilde{\mathcal{J}}(\xi)$ on the
quotient space:
$$\mathcal{W}(\xi):=\{\mathcal{U}(\xi)/\mathcal{S}(\xi)\}\otimes_{\mathbb{R}}\mathbb{C}\,.$$
Since $\tilde\eta\ne0$, $\mu$ is an eigenvalue of ${}^{g_N}\!\!\tilde{\mathcal{J}}(\xi)$.
By Equation (\ref{eqn-6}), $\mathcal{W}(\xi)=\tilde\xi\cdot\mathbb{R}\oplus\tilde\xi_1\cdot\mathbb{R}$. By Lemma \ref{lem-2.3},
${}^{2}\!\mathcal{J}(\xi)\xi=0$ and ${}^{2}\!\mathcal{J}(\xi)\xi_1={}^{2}\!\mathcal{J}(\xi)(\xi_1-\xi)\in\mathcal{S}(\xi)$ and thus
${}^2\!\tilde{\mathcal{J}}(\xi)=0$.
Since ${}^{g_{\tilde{\mathbb{C}}P}}\!\!\tilde{\mathcal{J}}(\xi)\tilde\xi=0$ and ${}^{g_{\tilde{\mathbb{C}}P}}\!\!\tilde{\mathcal{J}}(\xi)\tilde\xi_1=\tilde\xi_1$ we have
${}^{g_N}\!\!\tilde{\mathcal{J}}(\xi)\tilde\xi=0$ and ${}^{g_N}\!\!\tilde{\mathcal{J}}(\xi)\tilde\xi_1=\tilde\xi_1$. Thus $\mu\in\{0,1\}$. Assertion (2)
follows.

To prove Assertion (3), we note that ${}^{g_{\tilde{\mathbb{C}}P}}\!\!\mathcal{J}(\xi)=\frac14\id$ on $\mathcal{S}(\xi)$ and that
${}^{2}\!\mathcal{J}(\xi)$ is nilpotent and preserves
$\mathcal{S}(\xi)$. Assertion (4) follows from Assertions (1)-(3).
\end{proof}

\medbreak\noindent{\it Proof of Theorem \ref{thm-1.2}}. Let $(M,D)$ be an affine manifold which is affine
Osserman at $P\in M$. Choose local coordinates on $M$ so ${}^D\Gamma(P)=0$. Let
$D_0$ be the flat torsion-free connection defined on a neighborhood of $P$ whose Christoffel symbols vanish in
these coordinates. Set
$$D^\varepsilon:=\varepsilon
D+(1-\varepsilon)D_0$$ to define a $1$-parameter family of
metrics $g_N^\varepsilon$ interpolating between $g_N^1=g_N$ and
$g_N^0=g_{\tilde{\mathbb{C}}P}$. Since ${}^{D^\varepsilon}\!\!R(P)=\varepsilon\cdot{}^D\!R(P)$, all the connections
$D^\varepsilon$ are affine Osserman at
$P$. Thus Lemma \ref{lem-2.4} implies
$\operatorname{Spec}\{{}^{g_N^\varepsilon}\!\!\mathcal{J}(\xi)\}\subset\{0,1,\frac14\}$
for all $\varepsilon$ so the eigenvalue multiplicities are
unchanged during this perturbation as well. Taking $\varepsilon=0$ yields
the desired multiplicities and establishes Assertion (1) of
Theorem \ref{thm-1.2} for $\xi$ spacelike. We now use results of \cite{GKV} to see that spacelike Osserman implies
timelike Osserman and to relate the eigenvalues and eigenvalue
multiplicities on $S^+(T_QN,g_N)$ to the eigenvalues and eigenvalue
multiplicities on $S^-(T_QN,g_N)$; alternatively, of course, one could simply proceed directly as well. This proves Assertion
(1) of Theorem
\ref{thm-1.2}; Assertion (2) follows from Assertion (1).
\hfill\qed

\medbreak\noindent{\it Proof of Theorem \ref{thm-1.3}}.
Let $(M,D)$ be an affine manifold, let $Q\in N$, let $\pi\in\mathcal{P}_Q(N)$, and let $\xi\in S^+(\pi,g_N)$. Let
$a=\sigma_*\xi=\sigma_*\mathfrak{J}(\xi)$.
 By Lemma \ref{lem-2.2},
\begin{eqnarray*}
&&{}^{2}\!\mathcal{J}(\xi)=\left(\begin{array}{ll}{}^{D}\!\mathcal{J}(a)&0\\\star&{}^{D}\!\mathcal{J}(a)^t\end{array}\right),\quad
  {}^{2}\!\mathcal{J}(\mathfrak{J}\xi)=\left(\begin{array}{ll}{}^{D}\!\mathcal{J}(a)&0\\\star_1&{}^{D}\!\mathcal{J}(a)^t\end{array}\right),\\
&&{}^{2}\!\mathcal{J}(\pi)={}^{2}\!\mathcal{J}(\xi)-{}^{2}\!\mathcal{J}(\mathfrak{J}\xi)=
\left(\begin{array}{ll}0&0\\\star_\pi&0\end{array}\right)\,.
\end{eqnarray*}
Thus $\operatorname{Range}({}^{2}\!\mathcal{J}(\pi))\subset\mathcal{Y}$, ${}^{2}\!\mathcal{J}(\pi)\mathcal{Y}=0$,
and ${}^2\!\mathcal{J}(\pi)$ is nilpotent. Choose a basis $\{e_1,e_2,f_1,\dots,f_{2n-2}\}$ for $T_QN$
so $\operatorname{Span}\{e_1,e_2\}$ is the $+1$ eigenspace of
${}^{g_{\tilde{\mathbb{C}}P}}\!\!\mathcal{J}(\pi)$ and so $\operatorname{Span}\{f_1,\dots,f_{2n-2}\}$ is the
$+\frac12$ eigenspace of
${}^{g_{\tilde{\mathbb{C}}P}}\!\!\mathcal{J}(\pi)$. We compute:
$$({}^{g_N}\!\!\mathcal{J}(\pi)-\id)e_i={}^{2}\!\mathcal{J}(\pi)e_i\in\mathcal{Y}\quad\text{and}\quad
  ({}^{g_N}\!\!\mathcal{J}(\pi)-\textstyle\frac12\id)f_i={}^{2}\!\mathcal{J}(\pi)f_i\in\mathcal{Y}\,.$$
By Equation (\ref{eqn-5}),
$$
{}^{g_{\tilde{\mathbb{C}}P}}\!\!\mathcal{J}(\pi)=\id\text{ on }(\xi-\mathfrak{J}\xi)\cdot\mathbb{R}\quad\text{and}\quad
{}^{g_{\tilde{\mathbb{C}}P}}\!\!\mathcal{J}(\pi)=\textstyle\frac12\id\text{ on }\mathcal{S}(\xi)\,.
$$
Consequently by Lemma \ref{lem-2.3}, ${}^{g_{\tilde{\mathbb{C}}P}}\!\!\mathcal{J}(\pi)\mathcal{Y}\subset\mathcal{Y}$. Since
${}^{2}\!\mathcal{J}(\pi)=0$ on
$\mathcal{Y}$, this implies ${}^{g_N}\!\!\mathcal{J}(\pi)\mathcal{Y}\subset\mathcal{Y}$. We may now conclude:
$$\operatorname{Range}\{({}^{g_N}\!\!\mathcal{J}(\pi)-\id)\cdot({}^{g_N}\!\!\mathcal{J}(\pi)-\textstyle\frac12\id)\}\subset\mathcal{Y}\,.$$
Since ${}^{2}\!\mathcal{J}(\pi)=0$ on $\mathcal{Y}$, ${}^{g_N}\!\!\mathcal{J}(\pi)={}^{g_{\tilde{\mathbb{C}}P}}\!\!\mathcal{J}(\pi)$ on $\mathcal{Y}$ and consequently
by Equation (\ref{eqn-5}) and Lemma \ref{lem-2.3},
\begin{eqnarray*}
&&({}^{g_N}\!\!\mathcal{J}(\pi)-\id)\cdot({}^{g_N}\!\!\mathcal{J}(\pi)-\textstyle\frac12\id)\mathcal{Y}=\{0\}\quad\text{so}\\
&&({}^{g_N}\!\!\mathcal{J}(\pi)-\id)^2{\cdot}({}^{g_N}\!\!\mathcal{J}(\pi)-\textstyle\frac12\id)^2=\{0\}\,.
\end{eqnarray*}
Consequently $\operatorname{Spec}\{{}^{g_N}\!\!\mathcal{J}(\pi)\}\subset\{\frac12,1\}$ and ${}^{g_N}\!\!\mathcal{J}(\pi)$ has only $1\times 1$ or $2\times
2$ Jordan blocks. As in the proof of Theorem \ref{thm-1.2}, we set
$D^\varepsilon:=\varepsilon D+(1-\varepsilon)D_0$ to construct a $1$-parameter family of semi  para-complex Osserman metrics
$g_N^\varepsilon$ interpolating between $g^1_N=g_N$ and $g_N^0=g_{\tilde{\mathbb{C}}P}$. Since the eigenvalues are unchanged,
the eigenvalue multiplicities are unchanged. Consequently $\frac12$ is an eigenvalue of multiplicity $2n-2$ and $1$ is an
eigenvalue of multiplicity 2.\qed

\section{Examples}\label{sect-3} Throughout Section \ref{sect-3}, we will take $M=\mathbb{R}^n$ for some $n$ and we will
consider a point $Q\in Z(N)$ so that ${}^D\Gamma(P)=0$ where $P=\sigma(Q)$. This simplifies the
computations greatly. We will list only the possibly non-zero components of various tensors up to the obvious $\mathbb{Z}_2$
symmetries.

\begin{lemma}\label{lem-3.1}
Let $(M,D)$ be an affine manifold, let $Q\in Z(N)$, and let $P=\sigma(Q)$. Assume that ${}^D\Gamma(P)=0$.
\begin{enumerate}
\item The possibly non-zero curvatures of ${}^{g_D}\!R(Q)$  are ${}^{g_D}\!R_{ijkl^\prime}(Q)={}^{D}\!R_{ijk}{}^l(P)$.
\item The non-zero curvatures of ${}^{g_{\tilde{\mathbb{C}}P}}\!R(Q)$ are:
\begin{eqnarray*}
&&{}^{g_{\tilde{\mathbb{C}}P}}\!R_{i,i^\prime,i^{\prime},i}(Q)=-1,\\
&&{}^{g_{\tilde{\mathbb{C}}P}}\!R_{i,j^\prime,i^\prime,j}(Q)=
  {}^{g_{\tilde{\mathbb{C}}P}}\!R_{i,i^\prime,j^\prime,j}(Q)=-\textstyle\frac12\text{ for }i\ne j\,.
\end{eqnarray*}
\item ${}^{g_N}\!R(Q)={}^{g_{\tilde{\mathbb{C}}P}}\!
R(Q)+{}^{g_D}\!R(Q)$.
\end{enumerate}\end{lemma}

\begin{proof} Let $(u_1,\dots,u_n)$ be coordinates on a pseudo-Riemannian manifold $(U,g_U)$. Expand
$g_U=g_{ab}du^a\circ du^b$. Suppose that the $1$-jets of the functions $g_{ab}$ vanish at a point $S$ of
$U$. We then have
$$
{}^{g_U}\!R_{abcd}(S)=\textstyle\frac12\{\partial_{u_a}\partial_{u_c}g_{bd}+\partial_{u_b}\partial_{u_d}g_{ac}
-\partial_{u_a}\partial_{u_d}g_{bc}-\partial_{u_b}\partial_{u_c}g_{ad}\}(S)\,.
$$
We apply this observation to the setting at hand. Since $x^\prime$
and ${}^D\Gamma$ vanish at $Q$ and
$P$, respectively, the $1$-jets of $g_D$, of $g_{\tilde{\mathbb{C}}P}$, and of $g_N$ vanish at $Q$. We establish
Assertion (1) by computing:
\begin{eqnarray*}
&&{}^{g_D}\!R_{ijkl^\prime}(Q)=\textstyle\frac12\{\partial_{x_j}\partial_{x_{\ell^\prime}}(-2x_{h^\prime}{}^D\Gamma_{ik}{}^h)
-\partial_{x_i}\partial_{x_{\ell^\prime}}(-2x_{h^\prime}{}^D\Gamma_{jk}{}^h)\}(Q)\\
&&\qquad=\{\partial_{x_i}{}^D\Gamma_{jk}{}^l-\partial_{x_j}{}^D\Gamma_{ik}{}^l\}(P)={}^D\!R_{ijk}{}^l(P)\,.
\end{eqnarray*}
The proof of Assertion (2) and of Assertion (3) is similar.
\end{proof}

\medbreak\noindent{\it Proof of Theorem \ref{thm-1.4}}.
 Let $r\ge2$ and let $M:=\mathbb{R}^{r+1}$. Let $\{x_0,\dots,x_r\}$ be the usual coordinates on
$M$ and $\{x_{0^\prime},\dots,x_{r^\prime}\}$ the dual fiber coordinates on $T^*M$.  We let indices
$a,b,c,d$ range from
$1$ through $r$ and indices $i,j,k,l$ range from $0$ through $r$. Let
$U_a{}^b$ be a lower triangular matrix, i.e. $U_a{}^b=0$ for $b\le a$. Let $\theta=\theta(x_0)$ be a smooth function
of $1$ variable. Define a torsion-free connection $D$ on
$TM$ with non-zero Christoffel symbols:
$${}^D\Gamma_{0a}{}^b={}^D\Gamma_{a0}{}^b=\theta U_a{}^b\,.$$
The curvature is given by
$$
{}^D\!R_{ijk}{}^l=\partial_{x_i}\{{}^D\Gamma_{jk}{}^l\}-\partial_{x_j}\{{}^D\Gamma_{ik}{}^l\}
+\{{}^D\Gamma_{ic}{}^l\}\{{}^D\Gamma_{jk}{}^c\}
-\{{}^D\Gamma_{jc}{}^l\}\{{}^D\Gamma_{ik}{}^c\}\,.
$$
Without loss of generality, we suppose $i<j$. The first term can play a role only if $i=k=0$.
The second term plays no role.
The third can play a role only if
$i=k=0$. The final term plays no role.  Thus possibly non-zero curvatures are:
$$
\begin{array}{l}
{}^D\!R_{0a0}{}^b=\partial_{x_0}\{{}^D\Gamma_{a0}{}^b\}+\{{}^D\Gamma_{0c}{}^b\}\{{}^D\Gamma_{a0}{}^c\}
=\partial_{x_0}\theta\cdot U_a{}^b+\theta^2\cdot U_c{}^bU_a{}^c.
\end{array}$$
Let $X\in T_PM$. As we must have $0<a<b$ in the above relation,
$${}^D\!\mathcal{J}(X)\partial_{x_i}\in\operatorname{Span}\{\partial_{x_{i+1}},\dots,\partial_{x_r}\}\,.$$
Consequently $(M,D)$ is affine Osserman. Assume $\theta(0)=0$ and $\partial_{x_0}\theta(0)=-1$. We set $P=0$ and take
$Q=(0,0)$. We may then apply Lemma \ref{lem-3.1} to see:
\begin{eqnarray*}
&&{}^{g_N}\!R(Q)={}^{g_{\tilde{\mathbb{C}}P}}\!R(Q)+{}^{g_D}\!R(Q),\\
&&{}^{g_N}\!R(\partial_{x_i},\partial_{x_{i^\prime}},\partial_{x_{i^\prime}},\partial_{x_i})(Q)=-1,\\
&&{}^{g_N}\!R(\partial_{x_i},\partial_{x_{j^\prime}},\partial_{x_{i^\prime}},\partial_{x_j})(Q)=
  {}^{g_N}\!R(\partial_{x_i},\partial_{x_{i^\prime}},\partial_{x_{j^\prime}},\partial_{x_j})(Q)=-\textstyle\frac12\quad(i\ne j),\\
&&{}^{g_N}\!R(\partial_{x_i},\partial_{x_j},\partial_{x_k},\partial_{x_{d^\prime}})(Q)={}^D\!R_{ijk}{}^d(P)\,.
\end{eqnarray*}
First, take $\xi_1:=\frac{1}{\sqrt{2}}(\partial_{x_1}+\partial_{x_{1^\prime}})$. Then ${}^{g_N}\!\!\mathcal{J}(\xi_1)={}^{g_{\tilde{\mathbb{C}}P}}\!\!\mathcal{J}(\xi_1)$ is diagonalizable and
exhibits trivial Jordan normal form; the curvature of
$D$ plays no role. Next, we consider $\xi_2:=\frac{1}{\sqrt{2}}(\partial_{x_0}+\partial_{x_{0^\prime}})$
and $\mathfrak{J}\xi_2=\frac{1}{\sqrt{2}}(\partial_{x_0}-\partial_{x_{0^\prime}})$. Then:
\medbreak\qquad
${}^{g_{\tilde{\mathbb{C}}P}}\!\!\mathcal{J}(\xi_2)\xi_2=0$,\qquad\qquad\quad\phantom{a..}
${}^{g_D}\!\!\mathcal{J}(\xi_2)\xi_2=0$,
\smallbreak\qquad
${}^{g_{\tilde{\mathbb{C}}P}}\!\!\mathcal{J}(\xi_2)\partial_{x_a}=\textstyle\frac14\partial_{x_a}$,\qquad\quad\phantom{...}
${}^{g_D}\!\!\mathcal{J}(\xi_2)\partial_{x_a}=U_a{}^b\partial_{x_b}$,
\smallbreak\qquad
${}^{g_{\tilde{\mathbb{C}}P}}\!\!\mathcal{J}(\xi_2)\mathfrak{J}\xi_2=\mathfrak{J}\xi_2$,\qquad\qquad\phantom{..}
${}^{g_D}\!\!\mathcal{J}(\xi_2)\mathfrak{J}\xi_2=0$,
\smallbreak\qquad
${}^{g_{\tilde{\mathbb{C}}P}}\!\!\mathcal{J}(\xi_2)\partial_{x_{a^\prime}}=\textstyle\frac14\partial_{x_{a^\prime}}$,\qquad\quad\phantom{.}
${}^{g_D}\!\!\mathcal{J}(\xi_2)\partial_{x_{a^\prime}}=U_b{}^a\partial_{x_{b^\prime}}$.\qed

\medbreak\noindent{\it Proof of Theorem \ref{thm-1.5}}.
Let $n\ge3$, let $M=\mathbb{R}^n$, let $P=0$, and let $Q=(0,0)$. Let $\theta=\theta(x_1)$ be a smooth function of $1$
variable. Let
$\theta_1:=\partial_{x_1}\theta$; we suppose
$\theta(0)=0$ and $\theta_1(0)\ne0$. Let $D$ be the affine
connection whose only non-zero Christoffel symbol is
${}^D\Gamma_{22}{}^3=\theta$. Since $\theta=\theta(x_1)$ and since
the only non-zero covariant derivative is
$D_{\partial_{x_2}}\partial_{x_2}=\theta\partial_{x_3}$, ${}^{D}\!\!\mathcal{J}$ is nilpotent and
$(M,D)$ is affine Osserman as
the only non-zero curvature is:
$${}^D\!R(\partial_{x_1},\partial_{x_2})\partial_{x_2}=\theta_1\partial_{x_3}\,.$$

By Lemma \ref{lem-3.1},
\begin{eqnarray*}
&&{}^{g_N}\!R(\partial_{x_i},\partial_{x_{i^\prime}},\partial_{x_{i^\prime}},\partial_{x_i})(Q)=-1,\\
&&{}^{g_N}\!R(\partial_{x_i},\partial_{x_{j^\prime}},\partial_{x_{i^\prime}},\partial_{x_j})(Q)=
  {}^{g_N}\!R(\partial_{x_i},\partial_{x_{i^\prime}},\partial_{x_{j^\prime}},\partial_{x_j})(Q)
=-\textstyle\frac12\quad(i\ne j),\\
&&{}^{g_N}\!R(\partial_{x_{3^\prime}},\partial_{x_2},\partial_{x_2},\partial_{x_1})(Q)=\theta_1\,.
\end{eqnarray*}
Let $\xi:=\frac{1}{\sqrt{2}}(\partial_{x_1}+\partial_{x_{1{^\prime}}})\in
S^+(T_QN,g_N)$;
$\mathfrak{J}\xi=\frac{1}{\sqrt{2}}(\partial_{x_1}-\partial_{x_{1^\prime}})$.
As ${}^{2}\!\mathcal{J}(\xi)={}^{2}\!\mathcal{J}(\mathfrak{J}\xi)=0$,
${}^{g_N}\!\!\mathcal{J}(\xi)={}^{g_{\tilde{\mathbb{C}}P}}\!\!\mathcal{J}(\xi)$
and
${}^{g_N}\!\!\mathcal{J}(\pi_\xi)={}^{g_{\tilde{\mathbb{C}}P}}\!\!\mathcal{J}(\pi_\xi)$
are diagonalizable. Next consider
\begin{eqnarray*}
&&\eta:=\textstyle\frac{1}{2}(\partial_{x_1}+\partial_{x_3}+\partial_{x_{1^\prime}}+\partial_{x_{3^\prime}})\in S^+(T_QN,g_N),\\
&&\mathfrak{J}\eta=\textstyle\frac{1}{2}(\partial_{x_1}+\partial_{x_3}-\partial_{x_{1^\prime}}-\partial_{x_{3^\prime}})\in
S^-(T_QN,g_N)\,.
\end{eqnarray*}
The only non-trivial components are provided by:
$$
{}^{g_D}\!\!\mathcal{J}(\eta)\partial_{x_2}=\textstyle\frac{1}{2}\theta_1\partial_{x_{2^\prime}},\quad
{}^{g_D}\!\!\mathcal{J}(\mathfrak{J}\eta)\partial_{x_2}=-\textstyle\frac{1}{2}\theta_1\partial_{x_{2^\prime}},\quad
{}^{g_D}\!\!\mathcal{J}(\pi_\eta)\partial_{x_2}=\theta_1\partial_{x_{2^\prime}}\,.
$$
Since
$\pi_2:=\operatorname{Span}\{\partial_{x_2},\partial_{x_{2^\prime}}\}$
is contained both in the $\frac14$ eigenspace of
${}^{g_{\tilde{\mathbb{C}}P}}\!\!\mathcal{J}(\eta)$ and in the $\frac12$
eigenspace of ${}^{g_{\tilde{\mathbb{C}}P}}\!\!\mathcal{J}(\pi_\eta)$,
this analysis shows that both ${}^{g_N}\!\!\mathcal{J}(\eta)$ and
${}^{g_N}\!\!\mathcal{J}(\pi_\eta)$ exhibit non-trivial Jordan normal
form.\qed

\section{The third Gray identity, integrability, flatness, and para-complex Osserman}\label{sect-4}

In Section \ref{sect-4.1}, we show $(M,D)$ is flat implies $(N,g_N,\mathfrak{J})$ is integrable and satisfies the third Gray
identity. In Section
\ref{sect-4.2}, we show $(N,g_N,\mathfrak{J})$ is integrable implies $(M,D)$ is flat. In Section \ref{sect-4.3}, we show
$(N,g_N,\mathfrak{J})$ satisfies the third Gray identity implies $(M,D)$ is flat. In Section \ref{sect-4.4}, we show
$\mathfrak{J}{}^{g_N}\!\!\mathcal{J}(\pi)={}^{g_N}\!\!\mathcal{J}(\pi)\mathfrak{J}$ for all $\pi\in\mathcal{P}(N)$ if and only if
$(N,g_N,\mathfrak{J})$ satisfies the third Gray identity. This will complete the proof of Theorem
\ref{thm-1.6}.

\subsection{Flat geometry}\label{sect-4.1}
If $(M,D)$ is flat, then $(N,g_N,\mathfrak{J})$ is
isomorphic to $\tilde{\mathbb{C}}P$ by Theorem \ref{thm-1.1}; $\tilde{\mathbb{C}}P$ is integrable and satisfies the third Gray
identity.

\subsection{Integrability}\label{sect-4.2}
Let
$(M,D)$ be an affine manifold. Suppose that the Nijenhuis tensor
$N_{{\mathfrak{J}}}$ of Equation (\ref{eqn-4}) vanishes for the manifold $(N,g_N,\mathfrak{J})$. Let $P\in M$. Choose local coordinates on $M$ so
that
${}^D\Gamma(P)=0$. Let
$Q\in\sigma^{-1}(P)$. Then:
\medbreak\qquad\qquad
$\mathfrak{J}\partial_{x_i}=\partial_{x_i}-\{x_{i^\prime}x_{a^\prime}-2x_{b^\prime}{}^D\Gamma_{ia}{}^b\}\partial_{x_{a^\prime}}$,
\smallbreak\qquad\qquad
$\mathfrak{J}\partial_{x_j}=\partial_{x_j}-\{x_{j^\prime}x_{c^\prime}-2x_{d^\prime}{}^D\Gamma_{jc}{}^d\}\partial_{x_{c^\prime}}$,
\smallbreak\qquad\qquad
$[\partial_{x_i},\partial_{x_j}]=0$,
\smallbreak\qquad\qquad
$\mathfrak{J}[\mathfrak{J}\partial_{x_i},\partial_{x_j}]
   =2x_{b^\prime}\partial_{x_j}{}^D\Gamma_{ia}{}^b\partial_{x_{a^\prime}}$,
\smallbreak\qquad\qquad
$\mathfrak{J}[\partial_{x_i},\mathfrak{J}\partial_{x_j}]=-2x_{b^\prime}\partial_{x_i}{}^D\Gamma_{ja}{}^b\partial_{x_{a^\prime}}$,
\smallbreak\qquad\qquad
$[\mathfrak{J}\partial_{x_i},\mathfrak{J}\partial_{x_j}]_Q=\{2x_{b^\prime}\partial_{x_i}{}^D\Gamma_{ja}{}^b
-2x_{b^\prime}\partial_{x_j}{}^D\Gamma_{ia}{}^b\}_Q\partial_{x_{a^\prime}}$
\smallbreak\qquad\qquad\qquad\qquad
$+\{x_{i^\prime}x_{a^\prime}\partial_{x_{a^\prime}}(x_{j^\prime}x_{c^\prime})-
 x_{j^\prime}x_{a^\prime}\partial_{x_{a^\prime}}(x_{i^\prime}x_{c^\prime})\}_Q\partial_{x_{c^\prime}}$.
\smallbreak\qquad\qquad$N_{\mathfrak{J}}(\partial_{x_i},\partial_{x_j})(Q)=
4x_{b^\prime}{}^DR_{ija}{}^b(P)\partial_{x_{a^\prime}}$.
\medbreak\noindent Since $N_{\mathfrak{J}}=0$, we conclude $(M,D)$ is flat. We
note that $N_{\mathfrak{J}}$ always vanishes on $Z(N)$. Thus for
this computation it is necessary to take $Q$ arbitrary.

\subsection{The third Gray identity}\label{sect-4.3}
Suppose that $(N,g_N,\mathfrak{J})$ satisfies the third Gray identity which is given in
Equation (\ref{eqn-3}). Let $Q\in Z(N)$ and let $P=\sigma(Q)$.
Choose coordinates on $M$ so ${}^D\Gamma(P)=0$. We apply Lemma
\ref{lem-3.1}. Since $\tilde{\mathbb{C}}P$ satisfies the third Gray
identity, we conclude ${}^{g_D}\!R$ satisfies the third Gray identity
at $Q$. Thus
\begin{eqnarray*}
&&{}^{g_D}\!R(\partial_{x_i},\partial_{x_j},\partial_{x_k},\partial_{x_{l^\prime}})(Q)=
  {}^{g_D}\!R(\mathfrak{J}\partial_{x_i},\mathfrak{J}\partial_{x_j},
  \mathfrak{J}\partial_{x_k},\mathfrak{J}\partial_{x_{l^\prime}})(Q)\\
&=&{}^{g_D}\!R(\partial_{x_i},\partial_{x_j},\partial_{x_k},-\partial_{x_{l^\prime}})(Q)=
-{}^{g_D}\!R(\partial_{x_i},\partial_{x_j},\partial_{x_k},\partial_{x_{l^\prime}})(Q)\,.
\end{eqnarray*}
Consequently ${}^{g_D}\!R(Q)=0$. By Lemma \ref{lem-3.1}, this implies ${}^D\!R(P)=0$. Since $Q$, and hence $P$, was
arbitrary,
$(M,D)$ is flat.

\subsection{The commutation relation $\mathfrak{J}\,{}^{g_A}\!\!\mathcal{J}(\cdot)={}^{g_A}\!\!\mathcal{J}(\cdot)\mathfrak{J}$}\label{sect-4.4}
The third Gray identity in the complex setting is crucial -- see,
for example, the discussion in \cite{BGG08,DV08,DLV09}; a purely algebraic
computation shows that this condition is equivalent to the condition that
${}^{g_A}\!\!\mathcal{J}(\pi)$ commutes with the almost complex
structure for every complex $2$-plane $\pi$. This computation
extends to show that an almost para-Hermitian manifold
$(A,g_A,\mathfrak{J})$ satisfies the third Gray identity if and only
if
${}^{g_A}\!\!\mathcal{J}(\pi)\mathfrak{J}=\mathfrak{J}\,{}^{g_A}\!\!\mathcal{J}(\pi)$
for all $\pi\in\mathcal{P}(A)$. This completes the proof
of Theorem \ref{thm-1.6}.\hfill\qed

\end{document}